\documentclass[11pt,centertags]{amsart}
\usepackage{amssymb}            % extra common ams symbols -- not autoloaded
\usepackage{mathrsfs}             % Adds Script Fonts for Curly scripts see \scr below
\usepackage{hyperref}           % click to jump to refernce in .dvi
\usepackage{calc}               % allows arithmetic calculation with quantities
\usepackage{enumerate}          % improved enumeration environment control
\usepackage{tikz-cd}			% Latest and greatest CD package
\usepackage{fancyhdr}           % super headers --recommended
\usepackage{url}                % appropriately display url's
\usepackage{gloss}              % some nice formatting specials
\usepackage{yhmath}             % yh math fonts
\usepackage{color}
\usepackage{graphicx}

\usepackage[normalem]{ulem}
\usepackage{todonotes}

\newtheorem*{main*}{Main Theorem}

\newtheorem{theorem}{Theorem}
\newtheorem{thm}{Theorem}
\newtheorem*{theorem*}{Theorem}

\newtheorem{prop}{Proposition}[section]

\newtheorem{lem}{Lemma}[section]

\newtheorem{cor}{Corollary}
\newtheorem*{corollary*}{Corollary}
\newtheorem*{cor*}{Corollary}

\newtheorem*{question*}{Question}
\newtheorem{conjecture}{Conjecture}
\newtheorem*{conjecture*}{Conjecture}

\theoremstyle{definition}

\newtheorem{prob}{Problem}

\newtheorem{defn}{Definition}
\newtheorem*{definition*}{Definition}

\theoremstyle{remark}
\newtheorem{remark}{Remark}[section]

\numberwithin{equation}{section}

\newcommand{\R}{\mathbb{R}}
\newcommand{\Z}{\mathbb{Z}}

\newcommand{\N}{\mathbb{N}}

\newcommand{\mc}{\mathcal}

\newcommand{\Ga}{\Gamma}

\DeclareMathOperator{\Id}{Id}
\DeclareMathOperator{\tr}{trace}

\DeclareMathOperator{\Ric}{Ric}
\DeclareMathOperator{\vol}{Vol}
\DeclareMathOperator{\Vol}{Vol}
\DeclareMathOperator{\Jac}{Jac}

\DeclareMathOperator*{\bary}{bar}
\DeclareMathOperator{\Minvol}{Minvol}
\DeclareMathOperator{\Minent}{Minent}

\DeclareMathOperator{\SL}{SL}

\DeclareMathOperator{\SO}{SO}

\DeclareMathOperator{\Sp}{\Sp}

\newcommand{\op}[1]{\operatorname{#1}}
\newcommand{\set}[1]{\left\{#1 \right\}}
\newcommand{\til}{\widetilde}

\newcommand{\floor}[1]{\left\lfloor #1 \right\rfloor}

\newcommand{\pa}{\partial }

\newcommand{\of}{\circ }
\providecommand{\to}{\longrightarrow }

\newcommand{\abs}[1]{\left\lvert #1 \right\rvert }
\newcommand{\norm}[1]{\left\Vert #1 \right\Vert }
\newcommand{\inner}[1]{\left\langle #1 \right\rangle }

\newcommand{\noi}{\noindent}

\newcommand{\cout}[1]{}
\definecolor{darkcyan}{rgb}{0. 0.65, 0.65}

\begin{document}

\author[Chris Connell]{Chris Connell$^\dagger$}
\thanks{$\dagger$ This work is supported by the Simons Foundation, under grant \#210442}
\author[Shi Wang]{Shi Wang}
%\author[]{} % For second author, keep adding authors this way
% One \thanks command per author

\title[Positivity of simplicial volume]{Positivity of simplicial volume for nonpositively curved manifolds with a {R}icci-type curvature condition}

\address{Indiana University} %one \address command per author
\email{connell@indiana.edu} % one \email command per author

\address{Indiana University} %one \address command per author
\email{wang679@iu.edu}
%\curraddr{}
%\urladdr{} % use \textasciitilde instead of ~ in URL
%\dedicatory{}
%\date{\today} % not standard to put at bottom
%\translator{}
%\keywords{}
%\subjclass[2000]{Primary 51R10; Secondary 51R12}

%\subjclass[2000]{Primary 53C21,53C23; Secondary 53C24}

\begin{abstract}
We show that closed manifolds supporting a nonpositively curved metric with
negative $(\floor{\frac{n}{4}}+1)$-Ricci curvature, have positive simplicial
volume. This answers a special case of a conjecture of Gromov. We also establish some related results concerning bounded cohomology and volume growth entropy for this new class of manifolds. 
\end{abstract}

\maketitle

%\markboth{Short author names}{Short title} % used for subsequent
%pages. less desirable alternative

\thispagestyle{empty} % turn off page numbering on title page

% Text begins Below
%**************************************************************************
\section[]{Introduction}

In the early 1980s,  Gromov (\cite{Gromov82}) and Thurston (\cite[Chapter
6]{Thurston77}) introduced the notion of {\em simplicial volume} for any closed,
connected and orientable manifold $M$. The simplicial volume, denoted
$\norm{M}$, is a nonnegative real number which measures how efficiently the
fundamental class of $M$ can be represented by real singular cycles. It is
defined to be,
\[
\norm{M}=\inf \set{\sum_i \abs{a_i}\,:\, \left[\sum_i a_i \sigma_i\right]=[M]\in H_n(M,\R) },
\]
where the infimum is taken over all singular cycles with real coefficients
representing the fundamental class in the top homology group of $M$.

One reason for the importance of this invariant stems from its sitting at the
bottom of a chain of inequalities relating it to other invariants defined in
terms of geometric and dynamical properties of the manifold. Thus when it does
not vanish, there are numerous implications beyond the topology of $M$.

The simplicial volume has been shown to vanish for several large classes of manifolds. These include those admitting a nondegenerate circle action, or more generally a polarized $\mc{F}$-structure \cite{Fukaya87a,Cheeger90,Gromov82}, certain affine manifolds (\cite{Bucher16a}), and manifolds with amenable fundamental group (\cite{Gromov82}). When nonzero, the exact value of the simplicial volume has been calculated only for a few cases including hyperbolic manifolds \cite{Thurston77,Gromov82}), products of surfaces (\cite{Bucher-Karlsson08a}) and Hilbert modular surfaces (\cite{Loeh09}). More broadly, positivity of  simplicial volume has been established for only a few special classes of manifolds including higher genus surface-by-surface bundles (\cite{Hoster01}), negatively curved manifolds (\cite{Gromov82}), certain generalized graph manifolds (\cite{Connell-Serrato17}), and most closed locally symmetric spaces of higher rank  and some noncompact finite volume ones as well (\cite{Lafont06b,Bucher-Karlsson07a,Kim12,Loeh09}). Some connections to other invariants have been established (see e.g. \cite{Kuessner04,Kim15,Loeh09a,Loeh16,Bucher09,Bucher14,Bucher14a}), though calculations remain difficult.

The situation for admitting a nonpositively curved metric remains delicate. Some of these fall under the scope of the afore-mentioned results (e.g. admitting polarized $\mc{F}$-structures, certain affine structures, certain graph structures, higher rank locally symmetric structures, etc...), and as such vanishing or positivity of the simplical volume is known. Among the remainder, the following conjecture has been attributed to Gromov (\cite{Savage82} and see also \cite[p.11]{Gromov82}):
\begin{conjecture}
Any closed manifold of nonpositive curvature and negative Ricci curvature has $\norm{M}>0$.
\end{conjecture}

\begin{remark}
Note that the nonpositive curvature assumption is necessary here as the 3-sphere admits a metric of negative Ricci curvature (\cite{Brooks89}), and yet has vanishing simplicial volume.
\end{remark}

Nonpositively curved Riemannian manifolds can be classified by their geometric
rank, which is the minimum dimension of parallel Jacobi fields along geodesics.
The higher rank manifolds turn out to have universal covers which are either
metric products or symmetric spaces of noncompact type
(\cite{Ballmann85d,Burns87}), and hence their simplicial volume is already
understood. The remaining class of geometric rank one manifolds is somewhat
mysterious and includes manifolds with both vanishing and nonvanishing
simplicial volume. % By the deep work
% of Knieper \cite{Knieper} we know they have in a sense few flats.

In this paper, we consider the simplicial volume of arbitrary closed manifolds
admitting a nonpositively curved metric with an additional negative Ricci-type bound. More specifically, we establish a special case and weakened form of Gromov's conjecture under the
condition that we replace $\op{Ric}<0$ with $\op{Ric}_{\floor{\frac{n}{4}}+1}<0$ where,

\begin{defn}\label{def:k-ricci}
For $u,v\in T_xM$,
\[
\op{Ric}_{k}(u,v)=\sup_{\stackrel{V\subset T_xM}{\dim V=k}} \op{Tr}
R(u,\cdot,v,\cdot)|_V,
\]
where $R(u,\cdot,v,\cdot)|_V$ is the restriction of the curvature tensor to
$V\times V$, and thus the trace is with respect to any orthonormal basis of $V$.
Lastly, we set $\op{Ric}_{k}=\sup_{v\in T^1M}\op{Ric}_{k}(v,v)$.
\end{defn}

Namely, we show:

\begin{thm}\label{thm:main} Let $M$ be an oriented closed manifold of dimension $n$ admitting a
Riemannian metric of nonpositive curvature with $\op{Ric}_{\floor{\frac{n}{4}}+1}< 0$. Then the
simplicial volume satisfies $\norm{M}> 0$.
\end{thm}

%The Main Theorem represents a special case of and hence evidence toward this conjecture.
\begin{remark}
Note that the Ricci condition is necessary both in the conjecture and in Theorem
\ref{thm:main} since any manifold that splits as a product with $S^1$, e.g.
$M=N\times S^1$ for a nonpositively curved manifold $N$ is nonpositively curved
but has $\norm{M}=0$. More generally, any manifold with a polarized
$\mc{F}$-structure has $\norm{M}=0$ (see \cite{Cheeger90}). There are many
manifolds of nonpositive curvature admitting such an $\mc{F}$-structure, e.g.
the famous twisted product of surface by circle examples (see
\cite{Cheeger86a}).

Also, for any metric on $M$ with $\op{Ric}_{\floor{\frac{n}{4}}+1}< 0$ we may
compute an explicit bound for the simplicial volume in terms of the dimension
$n$ and its curvature tensor (up to second covariant derivatives). In principal
a bound for the simplicial volume can be obtained by taking the supremum of this
bound over all metrics, but we have not attempted to estimate this.
\end{remark}

\begin{remark}
Examples of spaces admitting $\op{Ric}_{\floor{\frac{n}{4}}+1}< 0$ can be
constructed by deforming negatively curved manifolds, e.g. in the neighborhood
of a closed geodesic for instance such that a few tangent directions have
vanishing curvature with the remainder being negative.

More interesting, and quite general, examples which do no admit negatively curved metrics, can be constructed as follows. Start with any closed manifold $N$ of nonpositive curvature and dimension $k$. Using the main result of \cite{TamNguyenPhan13}, we obtain a $k+1$-manifold $N_1$ containing $N$ as a totally geodesic submanifold and whose only planes of $0$-curvature are tangent to $N=N_0\subset N_1$. (This latter fact was not explicitly stated in her theorem, but follows from the strict hyperbolization and warped product construction used in its proof.) Iterating this construction we produce a sequence $N_1,N_2,\dots$ manifolds with totally geodesic embeddings $N_i\subset N_{i+1}$. For some $j\leq 3k$ this process yields a manifold $M=N_j$ of dimension $n=k+j$ which has $\op{Ric}_{\floor{\frac{n}{4}}+1}< 0$, as desired.

The simplest nontrivial example starts with $N$ a flat 2-torus and iterates six times to produce a nonpositively curved $8$-manifold $M$ with negative $\op{Ric}_{3}$-curvature and a single 2-flat which is isolated in the sense of \cite{Hruska05}. In fact, they showed that nonpositively manifolds with isolated flats are relatively hyperbolic. Setting $\Ga=\pi_1(M)$ and $\Ga'=\pi_1(N)$, by \cite{Mineyev-Yaman} (see also \cite{Franceschini15}) their relative bounded cohomology groups $H^*_b(\Ga,\Ga')$ surject onto the ordinary relative cohomology groups for any $\pi_1(M)$-module coefficients. The comparison maps between the relative long exact sequence form a commutative diagram,

\begin{centering}
%\smaller\smaller
\begin{tikzcd}
\cdots \rar & H_b^{*-1}(\Ga') \rar \ar{d}{\eta_1} & H_b^*(\Ga,\Ga') \rar \ar{d}{\eta_2} & H_b^*(\Ga) \rar \ar{d}{\eta_3} & H_b^*(\Ga') \rar \ar{d}{\eta_4} 
%& H_b^{*+1}(\Ga,\Ga') \rar \ar{d}{\eta_5} 
& \cdots \\
\cdots \rar & H^{*-1}(\Ga') \rar  & H^*(\Ga,\Ga') \rar & H^*(\Ga) \rar  & H^*(\Ga') \rar  %& H^{*+1}(\Ga,\Ga') \rar   
& \cdots
\end{tikzcd}
\end{centering}

Since $\Ga'\cong\Z^2$ is amenable, $H_b^*(\Ga')=0$ and hence $\eta_1$ and $\eta_4$ are the zero maps, so $H_b(\Ga,\Ga')\cong H_b(\Ga)$ and since $H^*(\Ga')=0$ for $*\geq 3$, $\eta_3$ is surjective for $*\geq 3$. In particular, this yields an alternate proof that $\norm{M}>0$ in this case. 

More generally, note that even when $\Ga'=\pi_1(N)$ is not amenable, $H^*(\Ga')=0$ for $*>\dim(N)$. If $\Ga$ is hyperbolic relative to $\Ga'$ so that $\eta_2$ is surjective, then it follows that $\eta_3$ is surjective for all $*>\dim(N)$ and $\norm{M}>0$.  If for $i\geq 1$ we cut $N_i$ along $N_{i-1}$, we obtain a manifold $N_i'$ with two boundary components $\pa^+N_i$ and $\pa^-N_i$. We suspect that in many cases $\pi_1(N_i')$ will be (strongly) hyperbolic relative to its boundary groups $\pi_1(\pa^+N_i)\cong \pi_1(N_{i-1})\cong \pi_1(\pa^-N_i)$ for each $i\geq 1$. If so, then $\Ga=\pi_1(M)$ will be (strongly) hyperbolic relative to $\Ga'=\pi_1(N)$ by Corollary 1.4 of \cite{Osin04} as it arises as an iterated HNN-extension of the $\pi_1(N_i)$ starting from $\pi_1(N)$.

%However, in the more general case when $\Ga'$ is nonamenable we cannot use the 5-lemma to conclude $\eta_3$ is surjective unless $\eta_2$ and $\eta_4$ are isomorphisms and $\eta_1$ and $\eta_5$ are surjective, which is seldom the case.

Note that the existence of a closed flat in $M$ for some nonpositively curved
metric implies the existence of a $\Z^2$ subgroup in $\pi_1(M)$, which in turn
implies that $\pi_1(M)$ is not Gromov hyperbolic. In particular, such a smooth
manifold $M$, including many of these examples, cannot admit any negatively
curved metric. (In fact, the existence of any flat in a closed nonpositively
curved manifold is conjectured to imply the existence of a closed flat.)

Other interesting cases include starting with $N$ a closed locally symmetric space of higher rank. Combining this construction with gluing techniques across negatively curved portions of each $N_i$, we can construct many other types of interesting examples to which our theorems apply.
\end{remark}

Under stronger curvature bounds we obtain commensurately stronger statements about the bounded cohomology $H^*_b(M,\mathbb R)$ of these spaces. 
\begin{thm}\label{thm:bcohom}
Let $M$ be an closed oriented $n$-manifold admitting a
Riemannian metric of nonpositive curvature with $\op{Ric}_{k+1}<0$ for some $k\leq\floor{\frac{n}{4}}$, then the comparison map $\eta:H^*_b(M,\mathbb R) \rightarrow H^*(M,\mathbb R)$ is surjective when $*\geq 4k$.
\end{thm}

The case $k=\floor{\frac{n}{4}}$, i.e. surjectivity of the top comparison map, is equivalent to Theorem \ref{thm:main}.

We prove our first result using estimates for the Jacobian of the barycenter
map. Using these we can also prove an extension to the case of maps.

\begin{thm}\label{thm:vol-ent}
Let $M$ be a closed oriented nonpositively curved Riemannian $n$-manifold 
with $\op{Ric}_{\floor{\frac{n}{4}}+1}<0$. Then there is a scale-invariant constant
$C>0$ depending only on the metric of $\til{M}$ such that for any map $f:N\to M$ from a closed oriented Riemannian manifold $N$ we have 
 \[
 h(N)^{n}\vol(N)\geq C \abs{\deg f}h(M)^n\vol(M)
 \]
 where $h(N)$ is the volume growth entropy of $N$.
%Note Ric_1=0 always so the condition forces n\geq 4.
%Given any closed oriented Riemannian manifolds $N$ and $M$ of dimension $n$, and
%assume $M$ is nonpositively curved and has $\op{Ric}_{\floor{\frac{n}{4}}+1}<
%0$. There is a constant $C>0$, depending only on the metric of $\til{M}$ and
%scale invariant, such that if $f:N\to M$ is any continuous map, then
%\[
%h(N)^{n}\vol(N)\geq C \abs{\deg f}h(M)^n\vol(M)
%\]
%where $h(N)$ is the volume growth entropy of $N$.
\end{thm}

\begin{remark}
The constant $C>0$ in the theorem above can be chosen independently of $\til{M}$ if and only if there is no sequence of maps $f_i:N_i\to M_i$ between manifolds satisfying the hypotheses such that $\frac{\abs{\deg f_i} h(M_i)^n\vol(M_i)}{h(N_i)^n\vol(N_i)}\to \infty$. We know of no such sequence, but if $C$ can be chosen universally then we must have $C<1$ as can be seen by taking $M$ to be a hyperbolic 3-manifold and $N$ any other metric on $M$ and $f=\Id$. 

On the other hand, for a fixed topological manifold $M$ if there is no sequence of nonpositively curved metrics $(M,g_i)$ with $\op{Ric}_{\floor{\frac{n}{4}}+1}(g_i)<0$ and $h(M,g_i)^n\Vol(M,g_i)\to \infty$ then the supremum of constants $C$ such that the theorem still holds will be a topological invariant of $M$. 
%fix
%We note that in the theorem we also have $h(N)^{n}\vol(N)\geq \sup_{f,M}\set{C \abs{\deg f}h(M)^n\vol(M)}$ where the supremum is taken over all $M$ satisfying the hypotheses and maps $f:N\to M$. Since this supremum is some positive universal constant, we may take $C$ in the theorem to be a universal constant, provided there is not a sequence of Riemannian manifolds $\set{M_i}$ satisfying the hypotheses and maps $f_i:N\to M_i$, from some $N$, such that $\abs{\deg f_i}h(M_i)^n\vol(M_i)\to \infty$. While there are examples of nonpositively curved graph manifolds for which this happens, we are not aware of any such examples satisfying our hypotheses.
%(However, see Examples A.3 of \cite{Gallot88})

Lastly, the assumption that $N$ and $M$ be oriented can be removed provided we use the appropriate generalization of $\deg(f)$ (see e.g. \cite{Olum53}).
\end{remark}

In Section \ref{sec:straight} we outline the generalized straightening technique we are employing. Then we recall a few properties of Patterson-Sullivan measures on rank one manifolds in Section \ref{sec:PS}, and apply this to barycentric straightening in Section \ref{sec:barystraight}. The critical estimates needed are established in Section \ref{sec:Jacob} and used to prove Theorem \ref{thm:main} and then Theorem \ref{thm:vol-ent} in Section \ref{section:natural}. Finally, in Section \ref{sec:Apps} we will present the proof of Theorem \ref{thm:bcohom} and present some other direct corollaries and applications of these theorems to the co-Hopf property of groups and positivity of some related topological invariants.

\subsection*{Acknowledgements}
We would like to thank T. T\^{a}m Nguy$\tilde{\hat{\mathrm{e}}}$n Phan for pointing out to us the examples relying on her construction in \cite{TamNguyenPhan13}, and also the anonymous referee for pointing out \cite{Mineyev-Yaman}.

\section{Straightening method}\label{sec:straight}

To prove the positivity of simplicial volume of certain manifolds, we follow the
general approach of Thurston \cite{Thurston77}. The essential idea is that if a
manifold admits a straightening, then its simplicial volume is positive. For
simplicity, we follow the notion of \cite{Lafont06b} as a brief summary of
Thurston's approach.

\begin{defn}\label{def:straightening} (\cite{Lafont06b}) Let $\til{M}^n$ be the universal
cover of an $n$-dimensional manifold $M^n$. We denote by $\Gamma$ the fundamental group
of $M^n$, and by $C_{\ast}(\til{M}^n)$ the real singular chain complex of
$\til{M}^n$. Equivalently, $C_{k}(\til{M}^n)$ is the free $\mathbb{R}$-module
generated by $C^0(\Delta^k,\til{M}^n)$, the set of singular $k$-simplices in
$\til{M}^n$, where $\Delta^k$ is equipped with some fixed Riemannian metric. We say a
collection of maps $st_k:C^0(\Delta^k,\til{M}^n)\rightarrow C^0(\Delta^k,\til{M}^n)$
is a straightening if it satisfies the following conditions:
\begin{enumerate}
	\item the maps $st_k$ are $\Gamma$-equivariant,
	\item the maps $st_{\ast}$ induce a chain map $st_{\ast}:C_{\ast}(\til{M}^n,\mathbb
	R)\rightarrow C_{\ast}(\til{M}^n,\mathbb R)$ that is $\Gamma$-equivariantly chain
	homotopic to the identity,
	\item the image of $st_n$ lies in $C^1(\Delta^n, \til{M}^n)$, that is, the top
	dimensional straightened simplices are $C^1$,
	\item there exists a constant $C$ depending on $\til{M}^n$ and the chosen
	Riemannian metric on $\Delta^n$, such that for any $f\in C^0(\Delta^n, \til{M}^n)$,
	and corresponding straightened simplex $st_n(f):\Delta^n\rightarrow\til{M}^n$, there
	is a uniform upper bound on the Jacobian of $st_n(f)$:
	$$|\Jac(st_n(f))(\delta)|\leq C$$
	for all $\delta\in \Delta^n$.
\end{enumerate}
\end{defn}

\begin{thm}\label{thm:straightening} (\cite{Thurston77}\cite{Lafont06b}) If $\til{M}^n$ admits a straightening, then the
simplicial volume of $M$ is positive.
\end{thm}

Therefore, the question can be solved if we can find a straightening. It is
worth noting that in the above theorem, one need only a weaker condition than
$(4)$ of Definition \ref{def:straightening}: that there is a uniform upper bound on the volume--the intergral of Jacobian--among all top dimensional straightened
simplices. In fact, for a non-positively curved manifold, it is not difficult to
find a collection of maps that satisfies only $(1),(2)$ and $(3)$ of Definition
\ref{def:straightening}. For instance, given any simplex, we can inductively
take the geodesic coning of its ordered vertices, but it is not so clear whether
the straightened simplices have uniformly bounded Jacobian or volume. However,
in the context of locally symmetric spaces of non-compact type excluding
$\mathbb H^2$, and $\SL(3,\mathbb R)/\SO(3)$ (See also \cite{Bucher-Karlsson07a}
for another type of straightening in the case of $\SL(3,\mathbb R)/\SO(3)$),
Lafont and Schmidt \cite{Lafont06b} introduced the barycentric straightening to
show the positivity of simplicial volume, where the Jacobian estimate in
$(4)$ of Definition \ref{def:straightening} relies on previous work of Connell
and Farb \cite{Connell03,Connell03c,Connell17}.

In this article, we will apply the barycentric straightening to geometric rank one
manifolds. In higher rank symmetric spaces $X=G/K$ the Patterson-Sullivan
measures are $K$-invariant and are nicely supported on the Furstenberg boundary
of $X$, the Jacobian estimate then turns into a careful analysis on certain Lie
groups and Lie algebras, losing symmetries in geometric rank one, and
lacking tools from Lie theory, we can not transplant the estimate to the situation
we have. Instead we will relate the Hessian of Busemann functions with sectional
curvatures to make the eigenvalue matching. Our estimate allows some presence of
zero curvatures of the manifold, as long as they occur in small dimensional subspaces, that is,
for any direction $v$ in the tangent bundle, we allow up to
$\floor{\frac{n}{4}}+1$ many directions (including $v$ itself) to have
zero curvatures with $v$. We will show that, for manifolds satisfying our
\emph{negative $(\floor{\frac{n}{4}}+1)$-Ricci curvature condition}, the
barycentric straightening is a straightening in the sense of Definition
\ref{def:straightening}. As a result, all closed nonpositively curved manifolds
with negative $(\floor{\frac{n}{4}}+1)$-Ricci curvature have positive
simplicial volume.

\section{Patterson-Sullivan measures and barycenters}\label{sec:PS}

We will briefly discuss in this section the Patterson-Sullivan measures in
geometric rank one spaces, following the work of Knieper \cite{Knieper98}. See
also the original Patterson-Sullivan theory for Fuchian groups
\cite{Patterson76}, and in higher rank symmetric spaces \cite{Albuquerque99}.

We fix the notation and let $M$ be a compact nonpositively curved geometric rank one
manifold, $\til{M}$ the universal cover of $M$, and $\Gamma$ the fundamental group of
$M$. In \cite{Knieper98}, Knieper showed that there exists a unique family of finite Borel
measures $\{\mu_x\}_{x\in \til{M}}$ fully supported on $\partial \til{M}$, called the
Patterson-Sullivan measures, which satisfies:
\begin{enumerate}
\item $\mu_x$ is $\Gamma$-equivariant, for all $x\in \til{M}$,
\item $\frac{d\mu_x}{d\mu_y}(\theta)=e^{hB(x,y,\theta)}$, for all $x,y\in \til{M}$, and
$\theta\in \partial\til{M}$,
\end{enumerate}
where $h$ is the volume entropy of $M$, and $B(x,y,\theta)$ is the Busemann function of
$\til{M}$. Recall that, the Busemann function $B$ is defined by
$$B(x,y,\theta)=\lim_{t\rightarrow \infty}\big(d_{\til{M}}(y,\gamma_\theta(t))-t\big)$$
where $\gamma_\theta(t)$ is the unique geodesic ray from $x$ to $\theta$.

We fix a basepoint $O$ in $\til{M}$ and shorten $B(O,y,\theta)$ to just $B(y,\theta)$.
We note that for fixed $\theta\in \til{M}$ the Busemann function $B(x,\theta)$ is
convex on $\til{M}$, and the nullspace of its Hessian $DdB_{(x,\theta)}$ in direction
$v_{x\theta}$--connecting $x$ to $\theta$ have zero sectional curvature with
$v_{x\theta}$ (the converse is not true). Furthermore, if $\til{M}$ is assumed to be
Ricci negative, and if $\nu$ is any finite Borel measure fully supported on
$\partial\til{M}$, by taking the integral of $B(x,\theta)$ with respect to $\nu$, we
obtain a strictly convex function (see the lemma below)
$$x\mapsto \mathcal{B}_{\nu}(x):=\int_{\partial \til{M}}B(x,\theta)d\nu(\theta)$$

Hence we can define the barycenter $\op{bar}(\nu)$ of $\nu$ to be the unique
point in $\til{M}$ where the function attains its minimum. It is clear that
this definition is independent of the choice of basepoint $O$. The following
lemma shows why the above function is strictly convex.

\begin{lem}\label{lem:strictly convex} Following the above notion, if we assume
$\til{M}$ have strictly negative Ricci curvature, and $\nu$ be any finite Borel measure
that is fully supported on $\partial\til{M}$, then the function
	$$x\mapsto \mathcal{B}_{\nu}(x)$$
	is strictly convex.
\end{lem}
\begin{proof} It is equivalent to show that the Hessian
	$$\int_{\partial\til{M}}DdB_{(x,\theta)}(\cdot,\cdot)d\nu(\theta)$$
	is positive definite. To see this, let $u\in T^1_x\til{M}$ be an arbitrary unit
	vector, we claim there exists $\theta_0\in \partial\til{M}$ such that
	$v_{x\theta_0}$ is orthogonal to $u$, and $DdB_{(x,\theta_0)}(u,u)> 0$. If not,
	$DdB_{(x,\theta)}(u,u)=0$ for all $\theta$ with $v_{x\theta}$ orthogonal to $u$, this
	implies by comparison theorem that the sectional curvatures of the two planes spanned
	by $v_{x\theta}$ and $u$ are all $0$ (See also Theorem \ref{thm:key estimate} for an
	explicit estimate), which contradicts with the fact that the Ricci curvature in
	direction $u$ is strictly negative. Therefore, there exists $\theta_0\in
	\partial\til{M}$ so that $DdB_{(x,\theta_0)}(u,u)=\delta_0>0$. By continuity, there
	exists an open neighborhood $U$ of $\theta_0$, such that
	$DdB_{(x,\theta)}(u,u)>\delta_0/2$ for all $\theta\in U$. We hence have
	$$\int_{\partial\til{M}}DdB_{(x,\theta)}(u,u)d\nu(\theta)\geq
	\int_{U}DdB_{(x,\theta)}(u,u)d\nu(\theta)>\delta_0/2\cdot \nu(U)>0$$
	since $\nu$ has full support in $\partial \til{M}$. Note that $u$ is selected
	arbitrarily, so the Hessian
	$$\int_{\partial\til{M}}DdB_{(x,\theta)}(\cdot,\cdot)d\nu(\theta)$$
	is positive definite, and the function
	$$x\mapsto\mathcal{B}_{\nu}(x)$$
	is strictly convex.
	
\end{proof}

\section{Barycentric straightening}\label{sec:barystraight}
We now discuss the barycentric straightening introduced by Lafont and Schmidt \cite{Lafont06b}
(based on the barycenter method originally developed by Besson, Courtois, and Gallot
\cite{Besson95}). We follow the notation in the previous section, and we further denote by
$\Delta^k_s$ the standard spherical k-simplex in the Euclidean space, that is
$$\Delta^k_s=\Big\{(a_1,\ldots ,a_{k+1})\mid a_i\geq 0,
\sum_{i=1}^{k+1}a_i^2=1\Big\}\subseteq \mathbb{R}^{k+1},$$
with the induced Riemannian metric from $\mathbb{R}^{k+1}$, and with ordered vertices
$(e_1,\ldots,e_{k+1})$. Given any singular $k$-simplex $f:\Delta^k_s\rightarrow
\til{M}$, with ordered vertices
$V=(x_1,\ldots,x_{k+1})=\left(f(e_1),\ldots,f(e_{k+1})\right)$, we define the
$k$-straightened simplex
$$st_k(f):\Delta^k_s\rightarrow \til{M}$$
$$st_k(f)(a_1,\ldots ,a_{k+1}):=\op{bar}\left(\sum_{i=1}^{k+1}a_i^2\nu_{x_i}\right)$$
where $\nu_{x_i}=\mu_{x_i}/\norm{\mu_{x_i}}$ is the normalized Patterson-Sullivan measure at
$x_i$. We notice that $st_k(f)$ is determined by the (ordered) vertex set $V$, and we
denote $st_k(f)(\delta)$ by $st_V(\delta)$, for $\delta\in\Delta^k_s$.

\begin{prob} Is barycentric straightening a straightening in the sense of Definition
\ref{def:straightening}?
\end{prob}

The answer to this question is not known in general. Actually it is not difficult to see
the barycentric straightening satisfies $(1)$-$(3)$ of Definition \ref{def:straightening},
and the proof is similar to that of \cite[Property (1)-(3)]{Lafont06b}. Note that to check
$(3)$, the $C^1$ smoothness of top dimensional straightened simplices, we need Lemma \ref{lem:strictly convex} in order to apply the inverse function theorem.

To check $(4)$ that the Jacobian of top dimensional straightened simplices are uniformly
bounded, we estimate as follows (which is also similar to \cite[Property (4)]{Lafont06b}). For
any $\delta=\sum_{i=1}^{n+1}a_ie_i\in \Delta_s^{n}$, $st_n(f)(\delta)$ is defined to be
the unique point where the function
$$x\mapsto \int_{\partial
\til{M}}B(x,\theta)d\left(\sum_{i=1}^{n+1}a_i^2\nu_{x_i}\right)(\theta)$$
is minimized. Hence, by differentiating at that point, we get the 1-form equation
$$\int_{\partial\til{M}}dB_{(st_V(\delta),\theta)}(\cdot)d\left(\sum_{i=1}^{n+1}a_i^2\nu_{x_i}\right)(\theta)\equiv
 0$$
which holds identically on the tangent space $T_{st_V(\delta)}\til{M}$.
Differentiating in a direction $u\in T_\delta(\Delta_s^n)$ in the source, one obtains the $2$-form equation

\begin{align}\label{eqn:2-form}
\begin{split}
&\quad\sum_{i=1}^{n+1}2a_i\langle
u,e_i\rangle_\delta\int_{\partial\til{M}}dB_{(st_V(\delta),\theta)}(v)d\nu_{x_i}(\theta)\\
&+\int_{\partial\til{M}}DdB_{(st_V(\delta),\theta)}(D_\delta(st_V)(u),v)d\left(\sum_{i=1}^{n+1}a_i^2\nu_{x_i}\right)(\theta)\equiv 0
\end{split}
\end{align}
which holds for every $u\in T_\delta(\Delta_n^k)$ and $v\in T_{st_V(\delta)}(\til{M})$.

Now we define two positive semidefinite symmetric endomorphisms $H_\delta$ and $K_\delta$
on $T_{st_V(\delta)}(\til{M})$:
$$\langle
H_\delta(v),v\rangle_{st_V(\delta)}=\int_{\partial\til{M}}dB^2_{(st_V(\delta),\theta)}(v)d\left(\sum_{i=1}^{n+1}a_i^2\nu_{x_i}\right)(\theta)$$
$$\langle
K_\delta(v),v\rangle_{st_V(\delta)}=\int_{\partial\til{M}}DdB_{(st_V(\delta),\theta)}(v,v)d\left(\sum_{i=1}^{n+1}a_i^2\nu_{x_i}\right)(\theta)$$
Note from Lemma \ref{lem:strictly convex} that $K_\delta$ is positive definite. From
Equation (\ref{eqn:2-form}), we obtain, for $u\in T_\delta (\Delta _s^n)$ a unit vector
and $v\in T_{st_V(\delta)}(\til{M})$ arbitrary, the following
\begin{equation}\label{eqn:Q1-bounds-Q2}
\begin{split}
\big| \big\langle K_\delta(D_\delta(st_V)&(u)),v \big\rangle \big| =\abs{-\sum_{i=1}^{n+1}2a_i\langle
u,e_i\rangle_\delta\int_{\partial\til{M}}dB_{(st_V(\delta),\theta)}(v)d\nu_{x_i}(\theta)}
\\
& \leq \left(\sum_{i=1}^{n+1}\langle
u,e_i\rangle_\delta^2\right)^{1/2}\left(\sum_{i=1}^{n+1}4a_i^2\left(\int_{\partial\til{M}}dB_{(st_V(\delta),\theta)}(v)d\nu_{x_i}(\theta)\right)^2\right)^{1/2}\\
&\leq 2\left(\sum_{i=1}^{n+1}a_i^2\int_{\partial\til{M}} dB^2_{(st_V(\delta),\theta)}(v)d\nu_{x_i}(\theta)\int_{\partial\til{M}}1\, d\nu_{x_i}\right)^{1/2}\\
&=2 \inner{H_\delta(v),v }^{1/2}
\end{split}
\end{equation}
via two applications of the Cauchy-Schwartz inequality.

For points $\delta\in\Delta_s^n$ where $st_V$ is nondegenerate, we now pick orthonormal
bases $\{u_1,\ldots ,u_n\}$ on $T_\delta(\Delta_s^n)$, and $\{v_1,\ldots ,v_n\}$ on
$S\subseteq T_{st_V(\delta)}(\til{M})$. We choose these so that $\{v_i\}_{i=1}^n$ are
eigenvectors of $H_\delta$, and $\{u_1,\ldots, u_n\}$ is the resulting basis obtained by
applying the orthonormalization process to the collection of pullback vectors
$\{(K_\delta\circ D_\delta(st_V))^{-1}(v_i)\}_{i=1}^n$. By the choice of bases, the
matrix $(\langle K_\delta\circ D_\delta(st_V)(u_i),v_j\rangle)$ is upper triangular, so
we obtain
\begin{align*}
|\det(K_\delta)\cdot \Jac_\delta(st_V)| &=|\det(\langle K_\delta\circ
D_\delta(st_V)(u_i),v_j\rangle)|\\
& =\Big|\prod_{i=1}^n\langle K_\delta\circ D_\delta(st_V)(u_i),v_i\rangle\Big| \\
& \leq \prod_{i=1}^n 2 \langle H_\delta(v_i),v_i\rangle^{1/2} \\
&=2^n \det(H_\delta)^{1/2}
\end{align*}
where the middle inequality is obtained via Equation (\ref{eqn:Q1-bounds-Q2}).
Hence we get the inequality
$$|\Jac_\delta(st_V)|\leq 2^n\cdot \frac{\det(H_\delta)^{1/2}}{\det(K_\delta)}$$

In order to obtain uniform bounds on the Jacobian, we need a similar ``BCG type'' of
estimate to bound $\det(H_\delta)^{1/2}/\det(K_\delta)$ in the context of geometric rank one spaces. Essentially, what lies behind is the question of eigenvalue matching, that is, with small eigenvalues of
$K_\delta$ can we find enough eigenvalues of $H_\delta$ to cancel? The main difficulty of
the question in general occurs when the space admits large flats, which will result in
too many small eigenvalues of $K_\delta$. However, under the assumption that the manifold has negative
	$(\floor{\frac{n}{4}}+1)$-Ricci curvature, we can make the cancellation work and hence obtain a uniform bound on the
Jacobian. We will establish this in later sections (See Theorem \ref{thm:Jacobian
estimate}). We summarize the discussion above into the following proposition.

\begin{prop}\label{prop:reduction}
	Let $K_\delta$, $H_\delta$ be the two positive semi-definite symmetric forms defined
	as above (note $K_\delta$ is actually positive definite).
	Assume there exists a constant $C$ that only depends on $\til{M}$, with the
	property that
	$$\frac{\det(H_\delta)^{1/2}}{\det(K_\delta)}\leq C$$
	Then the quantity $|\Jac(st_V)(\delta)|$ is universally bounded -- independent of the
	choice of $(n+1)$-tuple of points $V\subset \til{M}$, and of the point $\delta \in
	\Delta_s^n$. Hence the barycentric straightening is a straightening.
\end{prop}

\section{Estimating the Jacobian}\label{sec:Jacob}
In this section, we establish Theorem \ref{thm:Jacobian estimate} as our key estimate towards both Theorem \ref{thm:main} and \ref{thm:vol-ent}, and consequently we prove Theorem \ref{thm:main}. We start by introducing and analyzing our negative Ricci curvature condition.

\begin{defn}\label{def:k-th trace} For any positive semi-definite linear endomorphism $A:\:V^m\rightarrow V^m$,
and for any $k=1,2,...,m$, we define the $k$-th trace of $A$, denoted by
$\textrm{Tr}_k(A)$, to be
	$$\inf_{V_k\subset V^m}\textrm{Tr}(A|_{V_k}),$$
	where $V_k$ is a $k$-dimensional subspace (not necessarily invariant under $A$) of $V^m$, and $A$ is viewed as a bilinear form when taking restrictions. Equivalently, it is the sum of
	$k$ least eigenvalues of $A$.
\end{defn}
\begin{prop}\label{prop:average} Let $X$ be a compact manifold, and $\mu$ a regular
probability measure on $X$. If $A$ maps $X$ continuously into the set of all $m\times m$
positive semi-definite matrices, then
	$$\textrm{Tr}_k\big(\int_X A(x) d\mu(x)\big)\geq \inf_{x\in X} \textrm{Tr}_k\big(A(x)\big).$$
\end{prop}
\begin{proof}
	For any $\epsilon>0$, by continuity of $\textrm{Tr}_k$, there exists an integer $N$,
	$x_1,...,x_N\in X$, and $a_1,...,a_N$ with $\sum_{i=1}^N a_i=1$, such that
	$$\textrm{Tr}_k\Big(\sum_{i=1}^N a_i A(x_i)\Big)< \textrm{Tr}_k\Big(\int_X A(x)
	d\mu(x)\Big)+\epsilon.$$
	On the other hand, there exists $V_k\subset \mathbb R ^m$ such that
	$$\textrm{Tr}_k\Big(\sum_{i=1}^N a_i A(x_i)\Big)=\textrm{Tr}\Big(\sum_{i=1}^N a_i
	A(x_i)\Big)\mid_{V_k},$$
	Furthermore,
	$$\textrm{Tr}\Big(\sum_{i=1}^N a_i A(x_i)\Big)\mid_{V_k}= \sum_{i=1}^N a_i
	\textrm{Tr}\Big(A(x_i)\mid_{V_k}\Big)\geq \inf_{x\in X} \textrm{Tr}_k\big(A(x)\big).$$
	This shows
	$$\textrm{Tr}_k\Big(\int_X A(x) d\mu(x)\Big)+\epsilon> \inf_{x\in X} \textrm{Tr}_k\big(A(x)\big),$$
	where $\epsilon$ can be arbitrarily small. Hence the proposition holds.
\end{proof}

\begin{defn} Given an $n$-dimensional Riemannian manifold $M$ with curvature tensor $R$,
for any $u\in T_xM$, we define a symmetric bilinear form $R_u(v_1,v_2)=-\langle
R(u,v_1)u, v_2\rangle$, where $v_1,v_2\in T_xM$. In particular, if the manifold is
non-positively curved, then $R_u$ defines a positive semi-definite symmetric form on
$T_xM$. Furthermore, we define the $k$-th Ricci curvature in direction $u$ (compare with Definition \ref{def:k-ricci}) as
	$$\textrm{Ric}_k(u)=-\textrm{Tr}_k(R_u).$$
Hence the $n$-th Ricci coincides with the standard Ricci curvature.
\end{defn}

\begin{lem}\label{lem:n/3 condition}
	Let $M$ be a closed manifold with non-positive curvature. Then the following
	conditions are equivalent.
\begin{enumerate}
	\item $\dim(\textrm{null}(R_v))\leq n/4$ for all $v\in T^1M$.
	\item $\forall v\in T^1_xM$, there exists a subspace $F_v\subset T_xM$ of dimension at
	least $3n/4$, such that $\langle v,F_v\rangle=0$ and $R_v(u,u)\geq C_0$ for all $u\in
	F_v$, where $C_0$ is some universal constant that only depends on $(M,g)$.
	\item $\forall v\in T^1_xM$, the $k$-th eigenvalue (in increasing order) of $R_v$ is
	at least $C_0$ when $k>n/4$, where $C_0$ is some universal constant that only depends
	on $(M,g)$.
	\item There exists $\delta>0$ that only depends on $(M,g)$, so that
	$$\inf_{v\in T^1M}\textrm{Tr}_k(R_v)\geq \delta$$
	when $k>n/4$.
	\item The manifold has strictly negative $k$-th Ricci when $k>n/4$. That is,
	$\textrm{Ric}_k(v)< 0$ for all $v\in T^1M$ and $k>n/4$, or equivalently, $Ric_{\floor{\frac{n}{4}}+1}<0$, as in Definition \ref{def:k-ricci}.
\end{enumerate}
\end{lem}

\begin{proof}
	It is easy to see, by compactness of $M$, $(4)$ and $(5)$ are equivalent. We show by
	the loop of implications
	$(4)\Rightarrow(3)\Rightarrow(2)\Rightarrow(1)\Rightarrow(4)$.
	
	\noi $(4)\Rightarrow(3)$: Take $V_k$ to be the span of the first $k$ eigenvectors of
	$R_v$, with associated eigenvalues $\lambda_1\leq \lambda_2\leq...\leq \lambda_k$. By
	Definition \ref{def:k-th trace},
	$\lambda_1+\lambda_2+...+\lambda_k=\textrm{Tr}(R_v|_{V_k})\geq \textrm{Tr}_k(R_v)
	\geq \delta$, for some constant $\delta$ that only depends on $(M,g)$. So
	$\lambda_k\geq \delta/k \geq \delta/n$, with constant $\delta/n$ only depending on
	$(M,g)$.
	
	\noi $(3)\Rightarrow(2)$: Take $F_v$ to be the span of the last $n-k+1$ eigenvectors (the
	$k,k+1,...,n$-th) of $R_v$, where $k$ is selected to be $\floor{\frac{n}{4}}+1$. Then
	$F_v$ is orthogonal to $v$ as $v$ is always in the null space of $R_v$ which
	corresponds to the first eigenvalue $0$. Note that $R_v(u)\geq \lambda_k\geq C_0$ for
	all $u\in F_v$, where $\lambda_k$ is the $k$-th eigenvalue of $R_v$. Also note that
	the dimension of $F_v$ is $n-k+1=n-\floor{\frac{n}{4}}$, which is at least $3n/4$.
	
	\noi $(2)\Rightarrow(1)$: For any $v\in T^1M$, by the property of $F_v$, $F_v\cap
	\textrm{null}(R_v)=0$. Therefore, $\dim(\textrm{null}(R_v))+\dim(F_v)\leq n$, hence
	$\dim(\textrm{null}(R_v))\leq n/4$.
	
	\noi $(1)\Rightarrow(4)$: We set $l=\floor{\frac{n}{4}}+1$, and denote $\lambda_k(v)$ be
	the $k$-th eigenvalue of $R_v$. By $(1)$, $\lambda_l(v)>0$ for all $v\in T^1M$. Since
	$\lambda_l(v)$ is continuous on $v$, and $T^1M$ is compact, there exists a universal
	constant $\delta>0$ such that $\lambda_l(v)\geq \delta$, hence for any $k>n/4$, we
	have
	$$\inf_{v\in T^1M}\textrm{Tr}_k(R_v)\geq \lambda_k(v)\geq \lambda_l(v)\geq \delta$$
\end{proof}

\begin{defn}
	We say a nonpositively curved manifold have \emph{negative
	$(\floor{\frac{n}{4}}+1)$-Ricci curvature} if it satisfies any of the five
	conditions above.
\end{defn}

In order to obtain an estimate in Proposition \ref{prop:reduction}, we will need to compare the eigenvalues of $H_\delta$ with $K_\delta$. As an integrand of $K_\delta$, the Hessian of Busemann functions are closely related to the sectional curvatures via the Jacobi equation. We recall that if $X$ is a simply connected, nonpositively curved manifold, then the Busemann functions are of class $C^2$ (see \cite{Heintze77}), so that the Hessian $DdB$ is a
continuous symmetric form on $X$. And by convexity of Busemann functions, they are positive semi-definite
everywhere. Moreover, we can bound them below by the sectional curvatures as follows.

\begin{thm}\label{thm:key estimate} Let $\til{M}$ be the universal cover of some closed
non-positively curved manifold $M$, $x\in \til{M}$ and $\theta\in \partial \til{M}$. If $Y_0\in T^1_x
\til{M}$ is any unit vector in the horocycle direction, that is, $Y_0\perp \gamma_{x\theta}'(0)$,
where $\gamma_{x\theta}(t)$ is the geodesic ray connecting $x$ and $\theta$, then there
exists a constant $C$ that depends on the norm of up to the second derivative of
curvature ($\norm{R}$, $\norm{\nabla R}$, and $\norm{\nabla^2 R}$), such that
	$$DdB_{(x,\theta)}(Y_0,Y_0)\geq C \Big(-\big\langle
	R(\gamma_{x\theta}'(0),Y_0)\gamma_{x\theta}'(0),Y_0\big\rangle\Big)^{3/2}$$
\end{thm}

\begin{proof}
	We extend $Y_0$ along the geodesic $\gamma_{x\theta}(t)$ to $Y(t)$, the unique stable
	Jacobi field with $Y(0)=Y_0$. Then the Hessian $DdB_{(x,\theta)}(Y_0,Y_0)$ is the same as
	the second fundamental form in direction $Y_0$ of the horosphere determined by $x$ and
	$\theta$, which is further equal to $-\langle Y(0), Y'(0)\rangle$ (see \cite{Heintze77}). We
	now take the second covariant derivative along the geodesic ray of $\langle Y(t),
	Y(t)\rangle$,
	\begin{align*}
	\big\langle Y(t),Y(t)\big\rangle''&=2\Big(\big\langle Y'(t),Y'(t)\big\rangle+\big\langle
	Y(t),Y''(t)\big\rangle\Big)\\
	&=2\Big(\norm{Y'(t)}^2+R_{\gamma_{x\theta}'(t)}\big(Y(t)\big)\Big)
	\end{align*}
	Integrating along the geodesic ray, we obtain
	$$2\Big(\lim_{t\rightarrow \infty}\big\langle Y(t),Y'(t)\big\rangle-\big\langle
	Y(0),Y'(0)\big\rangle\Big)=2\int_0^\infty\Big(\norm{Y'(t)}^2+R_{\gamma_{x\theta}'(t)}\big(Y(t)\big)\Big)dt$$
	Since $Y(t)$ is stable, $\norm{Y(t)}^2$ converges to a constant, so its derivative
	$2\langle Y(t),Y'(t)\rangle$ goes to $0$. Therefore, from the above equality, we
	obtain further that
\begin{align*}
	DdB_{(x,\theta)}(Y_0,Y_0)=-\big\langle
	Y(0),Y'(0)\big\rangle&=\int_0^\infty\Big(\norm{Y'(t)}^2+R_{\gamma_{x\theta}'(t)}\big(Y(t)\big)\Big)dt\\
    &\geq\int_0^\infty R_{\gamma_{x\theta}'(t)}\big(Y(t)\big)dt
\end{align*}
	To finish the proof of the theorem, we need the following lemma from calculus.
	
	\begin{lem}\label{lem:cal}Let $F$ be a $C^2$ function on $[0,\infty)$. If $F\geq 0$,
	and $F''$ is bounded above by $L\geq 0$ then there is a constant $C>0$ that depends on $L$ such
	that
		$$\int_0^\infty F(t)dt\geq C\cdot F(0)^{3/2}$$
	\end{lem}
	\begin{proof} First, we show $F(t)\geq \big(\sqrt{F(0)}-L't\big)^2$ on the interval
	$\big[0,\sqrt{F(0)}/L'\big]$, for some $L'$ depending only on $L$. If we denote $G(t)=F(t)-
	\big(\sqrt{F(0)}-L't\big)^2$, then $G(0)=0$, and $G\big(\sqrt{F(0)}/L'\big)=F\big(\sqrt{F(0)}/L'\big)\geq 0$.
	Moreover, $G''(t)=F''(t)-2L'^2$, so if we choose $L'>\sqrt{L/2}$, then $G''(t)<0$.
	Therefore $G$ is concave hence $G\geq 0$ on $\big[0,\sqrt{F(0)}/L'\big]$. Using this result
	and noting also that $F\geq 0$, we can estimate the integral
		$$\int_0^\infty F(t)dt\geq \int_0^{\sqrt{F(0)}/L'}
		\big(\sqrt{F(0)}-L't\big)^2dt=\frac{F(0)^{3/2}}{3L'}=C\cdot F(0)^{3/2}$$
		where $C$ is some constant depending on $L$.
	\end{proof}
	
	We continue with the proof of Theorem \ref{thm:key estimate}. If we can apply Lemma \ref{lem:cal} to the function $R_{\gamma_{x\theta}'(t)}(Y(t))$, then we get the
	inequality of the theorem. So it suffices to show that the second derivative of
	$R_{\gamma_{x\theta}'(t)}(Y(t))$ is bounded above. Therefore we compute and estimate (by
simply writing $\gamma'$, $Y$, and $Y'$ for $\gamma_{x\theta}'(t)$, $Y(t)$, and $Y'(t)$):
\begin{align*}
\Big|\big(R_{\gamma_{x\theta}'(t)}(Y(t))\big)''\Big|&=\Big|\big(\big\langle(\nabla
	R)_{\gamma'}(\gamma',Y)\gamma',Y\big\rangle+2\big\langle R(\gamma',Y)\gamma',Y'\big\rangle\big)'\Big|\\
	&=\Big|\big\langle(\nabla^2
	R)_{(\gamma',\gamma')}(\gamma',Y)\gamma',Y\big\rangle+4\big\langle(\nabla
	R)_{\gamma'}(\gamma',Y)\gamma',Y'\big\rangle\\
	&\quad+2\big\langle R(\gamma',Y')\gamma',Y'\big\rangle+2\big\langle R(\gamma',Y)\gamma', Y''\big\rangle\Big|\\
	&\leq C(\norm{Y}^2+\norm{Y'}^2)
\end{align*}
	where the last inequality uses the Jacobi equation $Y''=-R(\gamma',Y)\gamma'$, and $C$ is a constant depending on $\norm{R}$, $\norm{\nabla R}$ and $\norm{\nabla^2 R}$. We also note that since $Y$ is the stable Jacobi field,
	$\norm{Y(t)}$ is non-increasing along the geodesic ray and therefore $\norm{Y(t)}\leq
	\norm{Y(0)}=1$. Thus we only need to show $\norm{Y'}^2$ is bounded. However, we have
\begin{align*}
	|\langle Y'(t),Y'(t)\rangle'|=|2\langle Y'', Y'\rangle|&=2|\langle
	R(\gamma,Y)\gamma',Y'\rangle|\\
	&\leq C_1\sqrt{\langle R(\gamma',Y)\gamma',Y\rangle \langle R(\gamma',Y')\gamma',
	Y'\rangle}\\
	&\leq C_2\sqrt{\langle -R(\gamma',Y)\gamma',Y\rangle}\norm{Y'}\\
	&\leq C_2(\langle -R(\gamma',Y)\gamma',Y\rangle+\norm{Y'}^2)\\
	&=C_2\langle Y,Y'\rangle'
\end{align*}
	where the first inequality is due to the fact that $-R$ is positive semi-definite,
	the second uses the bound of $\norm{R}$, and the third inequality is the Cauchy-Schwarz
	inequality. Here $C_1$ is some universal constant, and $C_2$ is a constant depending
	on $\norm{R}$. Integrating the above inequality, we obtain, for any $0<t<s<\infty$,
	$$|\langle Y'(t),Y'(t)\rangle-\langle Y'(s),Y'(s)\rangle|\leq C_2|\langle Y(t),Y'(t)
	\rangle-\langle Y(s),Y'(s)\rangle|$$
	As we see earlier in the proof, $\langle Y(s),Y'(s)\rangle$ increases to $0$ as $s$
	approaches to $\infty$. Note that
	$$\int_0^\infty \norm{Y'(t)}^2dt\leq \int_0^\infty
	\norm{Y'(t)}^2+R_{\gamma_{x\theta}'(t)}(Y(t))dt<\infty$$
	so $\langle Y'(s),Y'(s) \rangle$ also goes to $0$ as $s$ approaches to $\infty$. So by
	taking $s\rightarrow \infty$, we have
	$$\norm{Y'(t)}^2\leq C_2 |\langle Y(t),Y'(t)\rangle |\leq -C_2\langle
	Y(0),Y'(0)\rangle=C_2 DdB_{(x,\theta)}(Y_0,Y_0)$$
	But by the comparison theorem the Hessian $DdB_{(x,\theta)}(Y_0,Y_0)$ is bounded above by
	some constant depending on $\norm{R}$. This shows $\norm{Y'}$ is bounded by some constant
	on $\norm{R}$, hence the second derivative of $R_{\gamma_{x\theta}'(t)}(Y(t))$ is
	bounded by some constant on $\norm{R}$, $\norm{\nabla R}$, and $\norm{\nabla^2 R}$, and in
	view of Lemma \ref{lem:cal}, we obtain the inequality of the theorem.
	
\end{proof}

\begin{cor}\label{cor:n/4 small} Under the assumption of Theorem \ref{thm:key estimate},
if $M$ has negative $(\floor{\frac{n}{4}}+1)$-Ricci curvature, then
	$$\textrm{Tr}_{k+1}(DdB_{(x,\theta)}(\cdot,\cdot))\geq C_0$$
	where $k=\floor{\frac{n}{4}}$ and $C_0$ depends on the negative $(\floor{\frac{n}{4}}+1)$-Ricci constant in Lemma \ref{lem:n/3 condition}, in particular it depends on $(M,g)$.
\end{cor}

\begin{proof} We choose an orthonormal frame $e_1, e_2,..., e_{k+1}$ of the $k+1$ least
eigenvectors of $DdB_{(x,\theta)}(\cdot,\cdot)$, so that
$$\textrm{Tr}_{k+1}(DdB_{(x,\theta)}(\cdot,\cdot))=\sum_{i=1}^{k+1}DdB_{(x,\theta)}(e_i,e_i).$$
 According to Theorem \ref{thm:key estimate} and H\"older's inequality, we obtain
	$$\sum_{i=1}^{k+1}DdB_{(x,\theta)}(e_i,e_i)\geq C\sum_{i=1}^{k+1}
	R_{v_{x\theta}}(e_i,e_i)^{3/2} \geq C'\Big(\sum_{i=1}^{k+1}
	R_{v_{x\theta}}(e_i,e_i)\Big)^{3/2}$$
	The negative $(\floor{\frac{n}{4}}+1)$-Ricci curvature condition implies (see Lemma \ref{lem:n/3 condition})
	$$\sum_{i=1}^{k+1} R_{v_{x\theta}}(e_i,e_i)\geq Tr_{k+1}(R_{v_{x\theta}})\geq C''$$
	for some constant $C''$ depending on $(M,g)$. Combining the above two inequalities
    establishes this corollary.
\end{proof}

 As a first step towards Theorem \ref{thm:Jacobian estimate}, we use the result of Theorem \ref{thm:key estimate} to obtain the following lemma, which compares pointwise the integrands of $H_\delta$ and $K_\delta$ -- the two positive semi-definite forms described in Section \ref{sec:barystraight}. We remark that the power $2/3$ in the following lemma (which actually traces back to lemma \ref{lem:cal}) directly leads to our imposed ``$n/4$'' condition. If this power can be improved to be closer to $1$, then the resulting $k$-Ricci condition could also be slightly weakened, but is still limited to an ``$n/3$'' condition.

\begin{lem}\label{lem:Busemann estimate}
	Suppose $M$ is a closed non-positively curved $n$-manifold with negative
	$(\floor{\frac{n}{4}}+1)$-Ricci curvature, and $\til{M}$ is its Riemannian
	universal cover. Let $x\in \til{M}$, and $\theta\in \partial \til{M}$. Then there
	is a constant $C$ that depends on $(M,g)$, such that for all $v\in T^1_xM$ and all
	$u\in F_v$ (where $F_v$ satisfies $(2)$ of Lemma \ref{lem:n/3 condition}), we have $$dB^2_{(x,\theta)}(u,u)\leq C\Big(DdB_{(x,\theta)}(v,v)\Big)^{2/3}$$
\end{lem}
\begin{proof}
	We decompose $v$ as $v_1+v_2$ where $v_1$ is in the direction of $v_{x\theta}$, and
	$v_2$ is in the orthogonal direction, and we denote $\alpha$ the angle between
	$v_{x\theta}$ and $v$. We note that if $\sin\alpha=0$, then $v$ and $v_{x\theta}$ are
	parallel, so $dB^2_{(x,\theta)}(u,u)=0$, the inequality holds automatically. 
	
	On the other hand, if $\sin\alpha\ne 0$ then by Theorem \ref{thm:key estimate} we can estimate
\begin{align*}
DdB_{(x,\theta)}(v,v)&=DdB_{(x,\theta)}(v_2,v_2)\geq (\sin^2\alpha)\Big(C\cdot
	R_{v_{x\theta}}(\frac{v_2}{\norm{v_2}},\frac{v_2}{\norm{v_2}})\Big)^{3/2}\\
&=\frac{C^{3/2}}{|\sin
	\alpha|}R_{v_{x\theta}}(v_2,v_2)^{3/2}\geq C^{3/2} R_{v_{x\theta}}(v,v)^{3/2}\\
&=C^{3/2}R_v(v_{x\theta},v_{x\theta})^{3/2}
\end{align*}
	We note that when restricted to the subspace $F_v$, $R_v$ have eigenvalues at least $C_0$
	according to Lemma \ref{lem:n/3 condition}, hence
	$$R_v(v_{x\theta},v_{x\theta})\geq C_0\cos^2\big(\angle(v_{x\theta},F_v)\big)\geq C_0
	\cos^2\big(\angle(v_{x\theta},u)\big)=C_0 dB^2_{(x,\theta)}(u,u),$$
	Combining the two inequalities, we obtain
	$$dB^2_{(x,\theta)}(u,u)\leq C\Big(DdB_{(x,\theta)}(v,v)\Big)^{2/3}$$
	for some constant $C$ that depends on $\norm{R}$, $\norm{\nabla R}$, $\norm{\nabla^2 R}$, and
	the Ricci constant $C_0$ of Lemma \ref{lem:n/3 condition}. In
	particular it depends solely on $(M,g)$. This completes the proof.
\end{proof}

\begin{thm}\label{thm:Jacobian estimate} Suppose $M$ is a closed non-positively curved
$n$-manifold with negative $(\floor{\frac{n}{4}}+1)$-Ricci curvature, and
$\til{M}$ is its Riemannian universal cover. Let $x\in \til{M}$, $\theta\in \partial
\til{M}$, and $\nu$ be any probability measure that has full support on $\partial
\til{M}$. Then
	there exists a universal constant $C$ that only depends on $(M,g)$, so that
	$$\frac{\det\Big(\int_{\partial
	\til{M}}dB^2_{(x,\theta)}(\cdot,\cdot)d\nu(\theta)\Big)^{1/2}}{\det\Big(\int_{\partial
	\til{M}}DdB_{(x,\theta)}(\cdot,\cdot)d\nu(\theta)\Big)}\leq C$$
\end{thm}

\begin{proof}
We follow the framework of \cite{Besson95}, and set
$K_{x,\theta}:=DdB_{(x,\theta)}(\cdot,\cdot)$,
$H_{x,\theta}:=dB^2_{(x,\theta)}(\cdot,\cdot)$, and $K:=\int_{\partial
\til{M}}K_{x,\theta}(\cdot,\cdot)d\nu(\theta)$, $H=\int_{\partial
\til{M}}H_{x,\theta}(\cdot,\cdot)d\nu(\theta)$. Let $0< \lambda_1\leq\lambda_2\leq...\leq \lambda_n$ be the eigenvalues of $K$, and let $v$ be the eigenvector corresponding to $\lambda_1$. Then there is a constant $C'$ depending on $(M,g)$ such that,
for any $u\in F_v$, we have
\begin{align*}
    H(u,u)&=\int_{\partial \til{M}}H_{x,\theta}(u,u)d\nu(\theta)\leq C'\int_{\partial
	\til{M}}K_{x,\theta}(v,v)^{2/3}d\nu(\theta)\\
&\leq C'\Big(\int_{\partial
	\til{M}}K_{x,\theta}(v,v)d\nu(\theta)\Big)^{2/3}=C'\lambda_1^{2/3}
\end{align*}
	where the first inequality is due to Lemma \ref{lem:Busemann estimate}, and the
	second is the H\"older inequality. Therefore Lemma \ref{lem:n/3 condition} applies, and we can find an orthonormal frame
	$e_1,e_2,...,e_{n-k}$ at $x$ so that $H(e_i,e_i)\leq C'\lambda_1^{2/3}$ for $1\leq
	i\leq n-k$, where $k=\floor{\frac{n}{4}}$. This implies that
	$$\textrm{Tr}_{n-k}(H)\leq \sum_{i=1}^{n-k} H(e_i,e_i)\leq (n-k)C'\lambda_1^{2/3}.$$
	If we further denote $\mu_1\leq \mu_2\leq...\leq \mu_n$ the eigenvalues of $H$, then,
	we have
	$$\mu_i\leq \textrm{Tr}_{n-k}(H)\leq (n-k)C'\lambda_1^{2/3},$$ for $1\leq i\leq n-k$, since $\textrm{Tr}_{n-k}(H)$ is the sum of $n-k$ least eigenvalues of $H$.
	Note also that the eigenvalues of $H$ are at most $1$ and $k\leq n/4$, so we can
	estimate the following
	$$\frac{(\det H)^{1/2}}{\det K}= \frac{\prod_{i=1}^n \mu_i^{1/2}}{\prod_{i=1}^n
	\lambda_i}\leq \frac{[(n-k)C'\lambda_1^{2/3}]^{\frac{n-k}{2}}}{\lambda_1^k
	\lambda_{k+1}^{n-k}}\leq \frac{C''}{\lambda_{k+1}^{n-k}}$$
	for some constant $C''$ depending on $(M,g)$. Finally, we can bound $\lambda_{k+1}$
	in the following way:
	$$\lambda_{k+1}\geq \frac{1}{k+1} \textrm{Tr}_{k+1}(K) \geq
	\frac{1}{k+1}\inf_{\theta\in \partial \til{M}}\textrm{Tr}_{k+1}(K_{x,\theta})\geq
	\frac{C_0}{k+1},$$
	where the second inequality is due to Proposition \ref{prop:average}, and the third
	inequality is by Corollary \ref{cor:n/4 small}. Therefore, we conclude by combining
	the above inequalities that
	$$\frac{(\det H)^{1/2}}{\det K}\leq C''\Big(\frac{k+1}{C_0}\Big)^{n-k}\leq C$$
	where $C$ is a constant depending on $(M,g)$. Or more specifically, $C$ depends on the dimension $n$, the bounds on $\norm{R}$, $\norm{\nabla R}$, $\norm{\nabla^2 R}$, and
	the corresponding Ricci constant $C_0$ of Lemma \ref{lem:n/3 condition}
\end{proof}

As a direct consequence, we prove Theorem \ref{thm:main}.

\vspace{.3cm}

\textbf{Proof of Theorem \ref{thm:main}:} According to Theorem \ref{thm:Jacobian estimate} and Proposition \ref{prop:reduction}, we see that the barycentric straightening is a straightening on $\til{M}$, if $M$ satisfies the negative $(\floor{\frac{n}{4}}+1)$-Ricci curvature condition. Therefore, by Theorem \ref{thm:straightening}, the simplicial volume $\norm{M}>0$.

\section{The natural maps}
\label{section:natural}

In this section we describe the natural map, an essential
ingredient in the method of Besson-Courtois-Gallot, which we will use to prove the volume estimate of Theorem \ref{thm:vol-ent}.

Let $M$ continue to denote a closed oriented nonpositively curved Riemannian manifold with negative $(\floor{\frac{n}{4}}+1)$-Ricci curvature, and let $N$ be an arbitrary closed oriented manifold of the same dimension. Assume $f$ is a continuous map from $N$ to $M$. Let $\phi$ denote a choice of lift of $f$ to universal covers, i.e. $\phi=\til{f}:\til{N}\to \til{M}$.
We will also denote the metric and Riemannian volume form on the
universal cover $\til{N}$ by $g$ and $dg$ respectively.  Then for each
$s>h(g)$ and $y\in \til{N}$ consider the probability measure $\mu_y^s$
on $\til{N}$ in the Lebesgue class with density given by
$$\frac{d\mu_y^s}{dg}(z)=\frac{e^{-sd(y,z)}}{\int_{N}
  e^{-sd(y,z)}dg}.$$
The $\mu_y^s$ are well defined by the choice of $s$.

Consider the push-forward $\phi_*\mu_y^s$, which is a measure on $\til{M}$.
Define $\sigma_y^s$ to be the convolution of $\phi_*\mu_y^s$
with the Patterson-Sullivan measure $\nu_z$ on $\pa\til{M}$.

In other words, for $U\subset\partial \til{M}$ a Borel set, define
$$\sigma_y^s(U)=\int_{\til{M}}\nu_z(U)d(\phi_*\mu_y^s)(z)$$

%Even though it may not be the case that $\Vert \nu_z \Vert=1$, since the same measure always appears in each term of our equations or in resulting fractions, we may assume that the total measures are normalized, that is,
%$$\Vert\sigma_y^s\Vert=\Vert\mu_y^s\Vert=1.$$

For each $s>h(M)$, $\sigma^s_y$ has full support since each $\nu_z$ does. Therefore we define maps
$\widetilde{F}_s:\til{N}\to \til{M}$ by
$$\widetilde{F}_s(y)= \bary({\sigma_y^s}).$$

Note that $\sigma_y^s$ may not be a probability measure since we do not necessarily have $\Vert \nu_z \Vert=1$. However the barycenter is invariant under scaling of the measure.

The equivariance of $\phi$ and of $\{\nu_z\}$ and $\mu_y^s$ implies that
$\widetilde{F}_s$ is also equivariant. Hence $\widetilde{F}_s$
descends to a map $F_s: N\to M$. It is easy to see that $F_s$ is
homotopic to $f$.

\begin{prop}
\label{prop:homotopy}
The map $\Psi_s\colon [0,1]\times N\to M$ defined by
$$\Psi_s(t,y)=F_{s+\frac t{1-t}}(y)$$
is a homotopy between
$\Psi_s(0,\cdot)=F_s$ and $\Psi_s(1,\cdot)=f$.
\end{prop}

\begin{proof}
  From its definitions, $\til{F}_s(y)$ is continuous in $s$ and $y$. Observe that for
  fixed $y$, $\lim_{s\to \infty}\sigma^s_y=\nu_{\phi(y)}$. If follows that $\lim_{s\to
  \infty}\til{F}_s(y)=\phi(y)$. This implies the proposition.
\end{proof}

As in \cite{Besson95}, the implicit function theorem together with Lemma \ref{lem:strictly convex} implies that $F_s$ is $C^1$, and we will estimate
its Jacobian.

\begin{thm}[The Jacobian Estimate]
\label{thm:JacF}
For all $s>h(g)$ and all $y\in N$ we have
$$\vert \Jac F_s(y)\vert\le C s^n$$
for some constant $C$, depending only on $\til{M}$.
\end{thm}

\begin{proof}

We obtain the differential of $F_s$ by implicit differentiation:
$$0=D_{x=F_s(y)}\mathcal{B}_{\sigma_y^s}(x)=\int_{\pa \til{M}}
d{B}_{(F_s(y),\theta)}(\cdot) d\sigma^s_y(\theta)$$

Hence as $2$-forms
\begin{gather*}
0=D_y D_{x=F_s(y)}\mathcal{B}_{\sigma_y^s}(x)=\int_{\pa \til{M}}
Dd{B}_{(F_s(y),\theta)}(D_y F_s(\cdot),\cdot) d\sigma^s_y(\theta)\\
 -s \int_{N}\int_{\pa \til{M}} d{B}_{(F_s(y),\theta)}(\cdot)
\inner{ \nabla_y d(y,z),\cdot} d\nu_{\phi(z)}(\theta) d\mu^s_y(z)
\end{gather*}

The distance function $d(y,z)$ is Lipschitz and $C^1$ off of the cut
locus which has Lebesgue measure 0. It follows from the Implicit
Function Theorem (see \cite{Besson98}) that $F_s$ is $C^1$ for $s>h(g)$.
By the chain rule,
$$\Jac F_s = s^n\frac{\det\left(\int_{N}\int_{\pa \til{M}}
    d{B}_{(F_s(y),\theta)}(\cdot) \inner{ \nabla_y
      d(y,z),\cdot} d\nu_{\phi(z)}(\theta) d\mu^s_y(z)
  \right)}{\det\left(\int_{\pa \til{M}}
    Dd{B}_{(F_s(y),\theta)}(\cdot,\cdot)
    d\sigma^s_y(\theta)\right)}$$

Dividing numerator and denominator by $\frac{1}{\norm{\sigma^s_y}}$ and applying H{\"o}lder's inequality to the numerator gives:

$$|\Jac F_s|\leq s^n
\frac{\det\left(\frac{1}{\norm{\sigma^s_y}}\int_{\pa \til{M}} d{B}_{(F_s(y),\theta)}^2 d\sigma^s_y(\theta)\right)^{1/2}
  \det\left(\int_{N} \inner{\nabla_y d(y,z),\cdot}^2 d\mu^s_y(z)\right)^{1/2}}
 {\det\left(\frac{1}{\norm{\sigma^s_y}}\int_{\pa \til{M}} Dd{B}_{(F_s(y),\theta)}(\cdot,\cdot) d\sigma^s_y(\theta)\right)}$$

Using that $\tr \inner{\nabla_y d(y,z),\cdot}^2 =\left|\nabla_y
  d(y,z)\right|^2=1$, except possibly on a measure $0$ set, we may
estimate
$$\det\left(\int_{N} \inner{\nabla_y d(y,z),\cdot}^2
  d\mu^s_y(z)\right)^{1/2}\leq \left(\frac{1}{\sqrt{n}}\right)^n$$

Therefore
\begin{align}
\label{eq:Jac1}
| \Jac F_s|\leq
\left(\frac{s}{\sqrt{n}}\right)^n\frac{\det\left(\frac{1}{\norm{\sigma^s_y}}\int_{\pa \til{M}} d{B}_{(F_s(y),\theta)}^2
 d\sigma^s_y(\theta)\right)^{1/2}}{\det\left(\frac{1}{\norm{\sigma^s_y}} \int_{\pa \til{M}} Dd{B}_{(F_s(y),\theta)}(\cdot,\cdot) d\sigma^s_y(\theta)\right)}
\end{align}

Applying Theorem \ref{thm:Jacobian estimate} completes the proof.

\end{proof}

Now we prove Theorem \ref{thm:vol-ent}, which we restate for convenience:
{
\def\thetheorem{\ref{thm:vol-ent}}
\begin{theorem}
Let $M$ be a closed oriented nonpositively curved Riemannian $n$-manifold 
with $\op{Ric}_{\floor{\frac{n}{4}}+1}<0$. Then there is a scale-invariant constant
$C>0$ depending only on the metric of $\til{M}$ such that for any map $f:N\to M$ from a closed oriented Riemannian manifold $N$ we have 
 \[
 h(N)^{n}\vol(N)\geq C \abs{\deg f}h(M)^n\vol(M)
 \]
 where $h(N)$ is the volume growth entropy of $N$.
%Given any closed Riemannian manifolds $N$ and $M$ of dimension $n$, and assume $M$ is
%nonpositively curved and has $\op{Ric}_{\floor{\frac{n}{4}}+1}< 0$. There is a
%constant $C>0$, depending only on the metric of $\til{M}$ and scale invariant, such that if $f:N\to M$ is any continuous map, then
%\[
%h(N)^{n}\vol(N)\geq C \abs{\deg f}h(M)^n\vol(M)
%\]
%where $h(N)$ is the volume growth entropy of $N$.
\end{theorem}
\addtocounter{theorem}{-1}
}

\begin{proof}
From degree theory note that $f$ is isotopic to a $C^1$-map, which we again denote by $f$, and we have
\begin{align*}
\abs{\op{deg}(f)}\Vol(M)&=\abs{\int_N f^*dg_0}=\abs{\int_N F_s^*dg_0}\\
&\leq \int_N \abs{\Jac F_s}dg \leq C_M s^n \Vol(N).
\end{align*}
The second equality follows from Proposition \ref{prop:homotopy} and the last inequality from
Theorem \ref{thm:JacF}. Finally we take $C_1=\frac{1}{C_M}$ and then take the limit as $s\to h(N)$ to obtain $h(N)^{n}\vol(N)\geq C_1 \abs{\deg f}\vol(M)$. Applying this inequality to the identity map and $N=M$ shows that $h(M)>0$. Hence we may take $C=\frac{C_1}{h(M)^n}$ to obtain the inequality of the theorem.

From the proof of the Jacobian estimate, we observe that $C_M$ only depends on the metric $\til{g}$ on $\til{M}$, which is diffeomorphic to $\R^n$. The same dependence is well known for $h(M)$. Also, the quantity $h(M)^n\vol(M)$ is scale invariant, and therefore if $C$ is not scale invariant we may choose it to be so by taking the supremum of its values over scalings of $M$.
\end{proof}

%\begin{remark}
%Observe that the above theorem implies that for any metric $g_0$ on $M$, $h(M,g_0)\vol(M,g_0)\geq \sup_{g} C(g)h(M,g_0)\vol(M,g_0)>0$. Hence if $C=C(M,g)$ cannot be chosen to be a topological invariant of $M$, then there must be a sequence of nonpositively curved metrics $(M,g_i)$ with $\op{Ric}_{\floor{\frac{n}{4}}+1}(g_i)< 0$ such that $h(M,g_i)\vol(M,g_i)\to \infty$.
%\end{remark}

\section{Applications}\label{sec:Apps}

In this section we provide some applications of the above theorems. First we remark that Gromov in \cite{Gromov82} showed a proportionality principle for manifolds $M$ and $N$ sharing a common universal cover. Namely,

\[
\frac{\norm{N}}{\Vol(N)}=\frac{\norm{M}}{\Vol(M)}.
\]

More generally, if $f:N\to M$ is any map of nonzero degree then since $f$ induces a map at the level of singular chains such that $f_*[N]=\deg(f)[M]$, we have that $\norm{M}>0$ implies $\norm{N}>0$.

Now we recall that the dual of equivariant chains on $\til{M}$ with $L^1$ coefficients are equivariant cochains on $\til{M}$ with $L^\infty$ coefficients. This leads to the following application.

\subsection{Bounded cohomology} The notion of bounded cohomology is defined in \cite{Gromov82}, and has been studied in various contexts. If $M$ is a closed manifold, then the positivity of the simplicial volume $\norm{M}$ is equivalent to the surjectivity of the top dimensional comparison map $\eta:H^n_b(M,\mathbb R) \rightarrow H^n(M,\mathbb R)$ \cite{Gromov82}. Hence our Theorem \ref{thm:main} immediately implies the following:

\begin{cor} If $M$ is an $n$-dimensional closed nonpositively curved manifold that has $\op{Ric}_{\floor{\frac{n}{4}}+1}<0$, then the comparison map
$$\eta:H^n_b(M,\mathbb R) \rightarrow H^n(M,\mathbb R)$$
is surjective. In particular, the bounded cohomology $H^n_b(M,\mathbb R) $ is nontrivial.
\end{cor}

One can then naturally ask: are there surjectivity results in lower degrees? The question is already quite interesting when $M$ is locally symmetric of higher rank, as a positive answer leads to Dupont's Conjecture
\cite{Dupont79}. This conjecture states that the comparison maps from the continuous bounded cohomology of the corresponding semisimple Lie group into the standard continuous cohomology is always surjective. Recently, the same method of barycentric straightening has been applied by Lafont and the second author \cite{Lafont-Wang:15} in symmetric spaces, to answer Dupont's Conjecture in high degrees. Showing surjectivity in dimension $p$ requires a uniform bound on the $p$-Jacobian of barycentrically straightened simplices.

%We further remark that with only minimal modifications to the proof of Theorem \ref{thm:Jacobian estimate}, under a more restricted $k$-Ricci condition, we also obtain a bound on the $p$-Jacobian of the straightening map for certain $p<n$ depending on $k$. An essentially identical argument to that of \cite{Lafont-Wang:15} can then be applied in our context where $M$ is geometric rank one to show surjectivity of the comparison maps in certain lower degrees. More specifically, we obtain the following:
The same result holds in our context as well. We restate this here from the introduction,
{
\def\thetheorem{\ref{thm:bcohom}}
\begin{theorem}
Let $M$ be an oriented closed manifold of dimension $n$ admitting a
Riemannian metric of nonpositive curvature with $\op{Ric}_{k+1}<0$ for some $k\leq\floor{\frac{n}{4}}$, then the comparison map $\eta:H^*_b(M,\mathbb R) \rightarrow H^*(M,\mathbb R)$ is surjective when $*\geq 4k$.
\end{theorem}
\addtocounter{theorem}{-1}
}
\begin{proof}
The proof of Theorem \ref{thm:Jacobian estimate} already works verbatim for the case that, for the same $k$ in the proof, $k\leq\floor{\frac{n}{4}}$. This implies that we obtain the same bounds on the determinant restricted to $4k$-dimensional subspaces. Consequently, the same proof yields a bound for the $p$-Jacobians of the straightening map for $p\geq 4k$.

An essentially identical argument to that of \cite{Lafont-Wang:15} can then be applied in our context where $M$ is geometric rank one to show surjectivity of the comparison maps in the degrees at least $4k$.
\end{proof}

\subsection{The co-Hopf property}

We recall the following.
\begin{defn}
A group $G$ has the {\em co-Hopf property,} in which case $G$ is called {\em co-Hopfian}, if every injective endomorphism $h:G\to G$ is also surjective.
\end{defn}

\begin{cor}
Let $N$ be any closed oriented manifold admitting a nonzero degree map to a closed oriented nonpositively curved manifold $M$ with negative $(\floor{\frac{n}{4}}+1)$-Ricci curvature. Then $\pi_1(N)$ is co-Hopfian.
\end{cor}

\begin{proof}
Suppose not, then there is an injective endomorphism $g_*:\pi_1(N)\to\pi_1(N)$ with proper image. Consider the covering map $g:N\to N$ corresponding to the subgroup $g_*\pi_1(N)<\pi_1(N)$. Since $N$ is closed $\deg(g)$ is finite and it divides the index $[\pi_1(N):g_*\pi_1(N)]$. Hence $\abs{\deg(g)}>1$. Since $N$ admits a map $f:N\to M$ of nonzero degree, for any $k\in \N$ we have a map $f\of g^k:N\to M$ with $\abs{\deg(f\of g^k)}=\abs{\deg(f)\deg(g)^k}\geq 2^k$, which contradicts Theorem \ref{thm:vol-ent} for sufficiently large $k$.
\end{proof}

Note that the above corollary applies to the case of $N=M$ as well using the identity map.

\subsection{The Minvol and Minent invariants}

The {\em minimal volume} $\Minvol(M)$ and {\em lower Ricci minimal volume} $\Minvol_{\Ric}(M)$  are closely related smooth topological invariants of $M$ defined as,
$$\Minvol(M)=\inf \set{\vol(M,g) \, | \, -1\leq K_g\leq 1}$$
and
$$\Minvol_{\Ric}(M)=\inf \set{\vol(M,g) \, | \, -(n-1)\leq \Ric_g},$$
where the infimums are over all $C^\infty$ Riemannian metrics $g$ on $M$ and $K_g$ (resp. $\Ric_g$)  represents the sectional (resp. Ricci) curvatures of $g$.

A similar invariant is the {\em minimal (volume growth) entropy} invariant,
$$\Minent(M)=\inf \set{h(M,g) \, | \, \vol(M,g)=1}$$

Note that since $h(M,g)^n\vol(M,g)$ is a scale invariant quantity,
\begin{align*}
\Minent(M)&=\inf_g \set{h(M,g)\vol(M,g)^{\frac1n}}\\
&=(n-1)\inf \set{\vol(M,g)^{\frac1n} \, | \, h(M,g)\leq n-1}.
\end{align*}
The Bishop volume comparison Theorem (see \cite{Chavel84}) implies that $h(g)\leq n-1$ whenever  $-(n-1)\leq \Ric_g$. The latter condition occurs whenever $-1\leq K_g$, and consequently we have the following relationships:
\[
\Minvol(M)\geq \Minvol_{\Ric}(M) \geq \left(\frac{\Minent(M)}{n-1}\right)^n.
\]

Theorem \ref{thm:vol-ent} then directly implies the following.
\begin{cor}
Let $N$ be any closed oriented manifold admitting a nonzero degree map to a closed oriented nonpositively curved manifold $M$ with negative $(\floor{\frac{n}{4}}+1)$-Ricci curvature. Then $\Minvol(N),\Minvol_{\Ric}(N)$ and $\Minent(N)$ are all positive.
\end{cor}

%\bibliographystyle{myalpha}
%%\bibliographystyle{amsplain}
%\bibliography{references}
\newcommand{\etalchar}[1]{$^{#1}$}
\def\cprime{$'$}
\providecommand{\bysame}{\leavevmode\hbox to3em{\hrulefill}\thinspace}
\providecommand{\MR}{\relax\ifhmode\unskip\space\fi MR }
\providecommand{\MRhref}[2]{%
  \href{http://www.ams.org/mathscinet-getitem?mr=#1}{#2}
}
\providecommand{\href}[2]{#2}

\end{document}